\documentclass{amsart}
\usepackage{amsrefs}

\usepackage{amsfonts}
\usepackage{amsmath}
\usepackage[arrow,matrix,tips, curve]{xy}
\usepackage{amssymb}
\usepackage{amsthm}
\usepackage{latexsym}
\usepackage{url}
\usepackage{verbatim}

\theoremstyle{plain}
\newtheorem{theorem}{Theorem}
\newtheorem{proposition}[theorem]{Proposition}
\newtheorem{lemma}[theorem]{Lemma}
\newtheorem{corollary}[theorem]{Corollary}

\theoremstyle{definition}
\newtheorem{definition}{\mdseries\scshape Definition}

\theoremstyle{remark}

\newtheorem{remark}{\mdseries\scshape Remark}

\newcommand{\set}[1]{\left\{#1\right\}}
\newcommand{\abs}[1]{\left\vert#1\right\vert}
\DeclareMathOperator{\sigzero}{d}
\newcommand{\Nfp}{F}

\newcommand{\Ntc}{T}
\newcommand{\Nha}{C}
\newcommand{\Nsp}{S}

\newcommand{\Z}{\mathbb{Z}}
\newcommand{\Zp}{\mathbb{Z}_p}
\newcommand{\Qp}{\mathbb{Q}_p}
\newcommand{\oneunit}[1]{\left\langle#1\right\rangle}

\begin{document}

\title[Discrete exponentials using $p$-adic methods]{Counting fixed
  points, two-cycles, and collisions of the discrete exponential
  function using $p$-adic methods}

\author{Joshua Holden}

\thanks{The first author would like to thank the Hutchcroft Fund at
  Mount Holyoke College for support and the Department of Mathematics
  at Mount Holyoke for their hospitality during a visit in the spring
  of 2010.}

\address{Department of Mathematics,
Rose-Hulman Institute of Technology,
Terre Haute, IN 47803, USA}

\email{holden@rose-hulman.edu}

\author{Margaret M. Robinson} 

\address{Department of Mathematics, 
Mount Holyoke College,
50 College Street, South Hadley,
  MA 01075, USA}

\email{robinson@mtholyoke.edu}


\begin{abstract}
  Brizolis asked for which primes $p$ greater than 3 does there exist
  a pair $(g,h)$ such that $h$ is a fixed point of the discrete
  exponential map with base $g$, or equivalently $h$ is a fixed
  point of the discrete logarithm with base $g$.  Zhang (1995) and
  Cobeli and Zaharescu (1999) answered with a ``yes'' for sufficiently
  large primes and gave estimates for the number of such pairs when
  $g$ and $h$ are primitive roots modulo $p$. In 2000, Campbell showed
  that the answer to Brizolis was ``yes'' for all primes.  The first
  author has extended this question to questions about counting fixed
  points, two-cycles, and collisions of the discrete exponential
  map.  In this paper, we use $p$-adic methods, primarily Hensel's
  lemma and $p$-adic interpolation, to count fixed points, two cycles,
  collisions, and solutions to related equations modulo powers of a
  prime $p$.

\end{abstract}

\maketitle

\section{Introduction}

The idea of counting fixed points of discrete exponential functions is
usually traced back to Demetrios Brizolis
(see~\cite[Paragraph~F9]{guy}), who asked whether, given a prime $p>3$,
there is always a pair $(g, x)$ such that $g$ is a primitive root
modulo $p$, $g, x \in \set{1, \dots, p-1}$, and
\begin{equation} \label{fp}
g^{x} \equiv x  \pmod{p} \enspace ?
\end{equation}
We can regard solutions to this equation as fixed points of a discrete
exponential function.  Wen-Peng Zhang (\cite{zhang}) proved that
the answer to Brizolis' question was always yes for sufficiently large
$p$, a result which was rediscovered independently by Cobeli and
Zaharescu in~\cite{cobeli-zaharescu}.  Mariana (Campbell) Levin proved
the result for all primes in~\cite{campbell-thesis}.  (See
also~\cite{levin_et_al}.)

Zhang (and independently Cobeli and Zaharescu) also provided a way of
estimating the number of pairs $(g,x)$ which satisfy the conditions above
and also have $x$ being a primitive root.  Specifically, if $N(p)$ is
the number of such pairs given a prime $p$, we have:

\begin{theorem}[Zhang, independently by Cobeli and Zaharescu]
    \label{thm1}
\begin{equation*}
\abs{N(p)  - \frac{\phi(p-1)^{2}}{p-1}} \leq
\sigzero(p-1)^{2}\sqrt{p}(1+\ln p),
\end{equation*}
where $\sigzero(p-1)$ is the number of divisors of $p-1$.
\end{theorem}

The first author, in~\cite{holden02, holden02a}
investigated the problem of counting the number of solutions to
Brizolis' conditions when $g$ and $x$ are not necessarily primitive
roots.  If $\Nfp(p)$ is the number of such pairs $(g, x)$, it was
conjectured that
\begin{equation*}
\Nfp(p) \sim (p-1)
\end{equation*}
as $p$ goes to infinity. It was proved by the first author and Pieter Moree
in~\cite[Thm.~4.9]{holden_moree} that this is true for a 
set of primes of positive relative density.  Bourgain, Konyagin, and
Shparlinski proved in~\cite{bourgain_et_al_1} that
the conjecture is true for a set of primes of relative density 1.  The
same authors proved in~\cite{bourgain_et_al_2}, that a weaker result, 
  $\Nfp(p) = O(p)$, is true for all $p$, and also that $\Nfp(p) \geq (p-1)
  - o(p)$ for all $p$.

This paper was motivated by the attempt to similarly count solutions $(g,x)$
to the equation 
\begin{equation} \label{fpe}
g^{x} \equiv x  \pmod{p^e}
\end{equation}
with $g, x \in \set{1, \dots, p^e}$, $p\nmid g$ and $p\nmid x$.  Based
on numerical evidence, we conjecture that the number of these
solutions is asymptotically equivalent to $p^{e-1}(p-1)$ as $p$ goes
to infinity, and furthermore that the number of solutions with $g
\equiv i$ modulo $p$ is asymptotically equivalent to $p^{e-1}$ for any
$i$ as $p$ goes to infinity.  We would expect that the techniques used
to prove the theorems above could also be applied to this case.


We then attempted to investigate the situation as $p$ is held fixed
and $e$ goes to infinity.  This led naturally to an examination of the
function $x \mapsto g^x$ where $g$ is fixed and $x$ ranges through the
$p$-adic integers $\Zp$, which is carried out in
Sections~\ref{interpolation} and~\ref{hensel}.  The (perhaps)
surprising discovery is what happens when we look for solutions $x$
to~\eqref{fpe} not in the set $\set{1, \dots, p^e}$ but but rather in
the ``correct'' set $\set{1, \dots, p^em}$, where $m$ is the
multiplicative order of $g$ modulo $p$.  We show in
Section~\ref{fixed} that the number of solutions in this more natural
setting is exactly what one would expect from our conjectures, with no
error term.  (In the case $e=1$, \cite{levin_et_al} observes that it
is easy to find fixed points outside the set $\set{1, \dots, p}$ but
does not explicitly count them.)  Lev Glebsky, in~\cite{glebsky},
proves a similar result to ours in the case where $m=p-1$ using a very
different method.\footnote{Our thanks to Igor Shparlinski for this reference.}

  The papers~\cite{holden02, holden02a, holden_moree} also
  investigated three related questions: the number  of
  two-cycles of the discrete exponential function, or solutions to
\begin{equation}\label{tc}
  g^{h} \equiv a \mod{p} \quad \text{and} \quad g^{a} \equiv h \mod{p} ,
\end{equation}
the number of solutions to a discrete self-power equation
\begin{equation}\label{sp}
  x^{x} \equiv c \mod{p}
\end{equation} for fixed $c$,
and the number of collisions of the discrete self-power function,
i.e., solutions to 
\begin{equation} \label{ha}
    h^{h} \equiv a^{a} \mod{p} .
\end{equation}
It was conjectured in these papers that the number of solutions
$\Ntc(p)$ to~\eqref{tc} with $1 \leq g, h, a \leq p-1$ and $h \not
\equiv a$ modulo $p$ was 
\begin{equation*}
\Ntc(p) \sim (p-1),
\end{equation*}
the number of solutions
$\Nsp(p; c)$ to~\eqref{sp} with $1 \leq x \leq p-1$ was
\begin{equation*}
\Nsp(p; c) \sim \sum_{d \mid \frac{p-1}{m}} \frac{\phi(dm)}{dm}
\end{equation*}
where $m$ is the order of $c$ modulo $p$,
 and the number of solutions $\Nha(p)$ to  \eqref{ha} with
$1 \leq h, a \leq p-1$ and $h \not \equiv a$ modulo $p$ 
was
\begin{equation*}
\Nha(p) \sim \sum_{m \mid p-1} \phi(m)\left( \sum_{d \mid \frac{p-1}{m}}
  \frac{\phi(dm)}{dm}\right)^2 
= \sum_{d \mid p-1}  \frac{J_2(d)}{d},
\end{equation*}
where $J_2 (n) =n^2 \prod_{p \mid n} (1-p^{-2})$ is Jordan's totient
function, which counts the number of pairs of positive integers all
less than or equal to $n$ that form a mutually coprime triple together
with $n$. Balog, Broughan, and Shparlinski (\cite{balog_et_al}) showed
the weaker statements that 
$\Nsp(p; c) \leq p^{1/3+o(1)}m^{2/3}$
and
$\Nsp(p; c) \leq p^{1+o(1)}m^{-1/12}$, and that $\Nha(p) \leq p^{48/25+o(1)}$.
No nontrivial theorems on $\Ntc(p)$ seem to be known up to this point,
although Glebsky and Shparlinski (\cite{glebsky_shparlinski}) prove
some relevant results when $g$ is held fixed.

In Section~\ref{twocycles}, we investigate the number of solutions to
the equations 
\begin{equation}\label{tce}
  g^{h} \equiv a \mod{p^e} \quad \text{and} \quad g^{a} \equiv h \mod{p^e} ,
\end{equation}
where $g$ is fixed and $h$ and $a$ are in $\set{1, \dots, p^em}$ with
much the same results as before.  We also indicate how to generalize
this to more equations.  (Some of these results are also proved
in~\cite{glebsky}.)  In Section~\ref{selfpower} we similarly
investigate the equation
\begin{equation}\label{spe}
  x^{x} \equiv c \mod{p^e}
\end{equation} for fixed $c$, and $x$ in $\set{1, \dots, p^e(p-1)}$,
  and in Section~\ref{collisions} we investigate the equation
\begin{equation} \label{hae}
    h^{h} \equiv a^{a} \mod{p^e}
\end{equation}
for $h$ and $a$ in  $\set{1, \dots, p^e(p-1)}$.

The use of the discrete exponential function $x \mapsto g^x
\bmod{p}$ for $g$ a primitive root is well known in cryptography; its
inverse is commonly referred to as the discrete logarithm and
computing it is one of the basic ``hard problems'' of public-key
cryptography.  (See, for example, \cite[Section~3.6]{handbook}.)
There are also uses of the function when $g$ is not a primitive root,
for example, in the Digital Signature Algorithm.  (See, e.g.,
\cite[Section~11.5]{handbook}.  Finally, a few cryptographic
algorithms involve the self-power function $x \mapsto x^x \bmod{p}$
--- notably variants of the ElGamal signature scheme, as noted
in~\cite[Note~11.71]{handbook}.  The security of these cryptographic
algorithms rely on the unpredictability of the inputs to these maps
given the outputs.  The results above and the ones in this paper go
some way toward reassuring us that these maps are in fact behaving as
if the inputs are randomly distributed given only basic facts known
about the outputs.

\section{Interpolation} \label{interpolation}

Let $g \in \Z$ be fixed and take $p$ an odd prime.  In order to count solutions to $g^x \equiv x
\pmod{p^e}$, the obvious first step would be to interpolate the
function $f(x) = g^x$, defined on $x \in \Z$, to a function on $x \in
\Zp$.  Unfortunately, this is not possible unless $g \in 1 + p\Zp$.
(See for example, \cite[Section 4.6]{gouvea}, or \cite[Section II.2]{koblitz}.)
However, if we ``twist'' the function slightly, then interpolation is
possible.

To do this, let $\mu_{p-1} \subseteq \Zp^\times$
be the set of all $(p-1)$-st roots of unity.  Then for odd prime $p$, we have the
Teichm\"uller character
$$\omega: \Zp^\times \to \mu_{p-1},$$
which is a surjective homomorphism.
It is known that $\Zp^\times$ has a canonical decomposition as $\Zp^\times \cong
\mu_{p-1} \times (1+p\Zp)$ \cite[Cor. 4.5.10]{gouvea}, and thus for $x$ in
$\Zp^\times$ we may uniquely write $x = \omega(x) \oneunit{x}$ for some
$\oneunit{x} \in 1 + p\Zp$.

\begin{proposition}[Prop. 4.6.3 of~\cite{gouvea}; see also Section II.2
  of~\cite{koblitz}]
\label{gouvea-interp}
For $p \neq 2$, let $g \in \Zp^\times$ and $x_0 \in \Z/(p-1)\Z$, and let
$$I_{x_0} = \set{x \in \Z \mid x \equiv x_0
  \pmod{p-1}} \subseteq \Z.$$
Then $$f_{x_0}(x) = \omega(g)^{x_0}\oneunit{g}^x$$ defines a function
on $\Zp$ such that $f_{x_0}(x) = g^x$ whenever $x\in I_{x_0}$.
\end{proposition}

In fact we can push this a little further:

\begin{proposition}
  Let $m$ be any multiple of the multiplicative order of $g$ modulo
  $p$, $p \neq 2$, such that $m \mid p-1$.
  Let $g \in \Zp^\times$ and $x_0 \in \Z/m\Z$, and let
$$I_{x_0} = \set{x \in \Z \mid x \equiv x_0
  \pmod{m}} \subseteq \Z.$$
Then $$f_{x_0}(x) = \omega(g)^{x_0}\oneunit{g}^x$$ defines a function
on $\Zp$ such that $f_{x_0}(x) = g^x$ whenever $x\in I_{x_0}$.
\end{proposition}

\begin{proof}
Since $g^m=1$, $\omega(g)^m = \omega(g^m) = 1.$ 
If $x_0, x_0' \in \Z/(p-1)\Z$ and $ x_0 \equiv x_0' \pmod{m}$,
then the two functions $f_{x_0}$ and $f_{x_0'}$ given by
Proposition~\ref{gouvea-interp} are equal and agree with $g^x$ on
$I_{x_0} \cup  I_{x_0'}$.
\end{proof}

Also, as noted for $p\neq 2$  in~\cite{gouvea}, these functions fit together into a
function on $\Zp \times \Z/m\Z$ defined by $F(x_1, x_0) =
f_{x_0}(x_1)$, such that if $x \in \Z$ and $x \equiv x_0 \pmod{m}$ we
have $F(x,x) = f_{x_0}(x) = g^x.$ 
Then we have a diagram:
$$\begin{xy}
\UseTips
\xymatrix{\Zp \times \Z/m\Z \ar[r]^(.6)F \ar[d] & \Zp^\times \ar[d]\\
\Z/p^e\Z \times \Z/mZ \ar[r]^(.6){\bar{F}} & (\Z/p^e\Z)^\times}
\end{xy}$$
where the vertical arrows are the natural surjections.   
This commutes as a consequence of the following lemma:

\begin{lemma}[Cor. 4.6.2 and just below of~\cite{gouvea} or Lemma 2.2.5 of~\cite{igusa-forms} ]
\label{one-unit-lemma}
For any positive integer $k$, $(1 + p\Zp)^k \subseteq 1+pk\Zp$.
\end{lemma}

The lemma implies that $\oneunit{g}^{p^e}
\equiv 1 \pmod{p^e}$, and therefore $\oneunit{g}^x
\equiv \oneunit{g}^{x'} \pmod{p^e}$ when $x \equiv x' \pmod{p^e}$. 
(Recall that
$\Zp/p^e\Zp$ is isomorphic to $\Z/p^e\Z$ for any $e$.)

For $p \neq 2$, if
we let $\Delta$ be
the diagonal inclusion map 
$$\Delta: \Z \to \Zp \times \Z/m\Z$$ given
by the canonical injection $\Z \hookrightarrow \Zp$ and the canonical
surjection $\Z \twoheadrightarrow \Z/m\Z$, then
the previous diagram extends nicely to:
$$\begin{xy}
  \UseTips
  \xymatrix{\ \Z\  \ar@{^{(}->}[rd]^(.4){\Delta} \ar[ddd] \\
& \Zp \times \Z/m\Z
    \ar[r]^(.6)F \ar[d] & \Zp^\times \ar[d]\\ 
& \Z/p^e\Z \times \Z/mZ \ar[r]^(.6){\bar{F}} & (\Z/p^e\Z)^\times \\
\Z/p^em\Z \ar@{<->}[ru]^(.4){\sim}_(.4){CRT} \\
}
\end{xy}$$
where the isomorphism is given by the Chinese Remainder Theorem.
Furthermore, the composition of the maps on the top line is just the
map $x \mapsto g^x$ and the composition across the bottom line is the
map $x \mapsto g^x \bmod p^e$:
$$\begin{xy}
  \UseTips
  \xymatrix{\ \Z\  \ar@{^{(}->}[rd]^(.4){\Delta} \ar[ddd] \ar@/^/@<1ex>[rrd]^{x
      \mapsto g^x}\\
& \Zp \times \Z/m\Z
    \ar[r]^(.6)F \ar[d] & \Zp^\times \ar[d]\\ 
& \Z/p^e\Z \times \Z/mZ \ar[r]^(.6){\bar{F}} & (\Z/p^e\Z)^\times  \\
\Z/p^em\Z \ar@{<->}[ru]^(.4){\sim}_(.4){CRT} \ar@/_/@<-1ex>[rru]_{x \mapsto g^x
  \bmod p^e} \\
}
\end{xy}$$
Therefore finding all solutions $(x_1, x_0)$ to $F(x_1, x_0) \equiv x_1 \pmod{p^e},$
which is the same as finding all solutions to
$f_{x_0}(x_1) \equiv x_1 \pmod{p^e}$ for all possible $x_0 \in \Z/m\Z$, will
give us 
all solutions to $g^x \equiv x \pmod{p^e}$ as $x$ ranges over
$\Z/p^em\Z$.
\section{Hensel's Lemma} \label{hensel}

\begin{definition}[Defn. III.4.2.2
  of~\cite{bourbaki}] 
  A power series $f(x_1,x_2,\dots,x_n)$ in the ring of formal power
  series $\mathbb{Z}_p[[x_1,\dots,x_n]]$ with coefficients in
  $\mathbb{Z}_p$ is called \emph{restricted} if $f(x_1,\dots,x_n)=
  \sum_{(\alpha_i)} C_{\alpha_1,\alpha_2, \cdots, \alpha_n}
  x_1^{\alpha_1}\cdots x_n^{\alpha_n}$ and for every neighborhood $V$
  of 0 in $ \Zp $ there is only a finite number of coefficients
  $C_{\alpha_1,\alpha_2, \cdots, \alpha_n}$ not belonging to $V$ (in
  other words, the family $(C_{\alpha_1,\alpha_2, \cdots, \alpha_n})$
  tends to 0 in $\Zp$).


In particular, the series in this paper are going to be such that
$C_{0,0,\dots,0} \in \mathbb{Z}_p$ and $C_{\alpha_1,\alpha_2,
  \dots,\alpha_n} \in p^{\alpha_1+\alpha_2+
   \cdots+ \alpha_n-1}\mathbb{Z}_p$ when $ \alpha_1+\alpha_2+ \cdots+
 \alpha_n>0.$
\end{definition}

In this section, we include two versions of Hensel's lemma.  The first
version is for $n$ restricted power series in $n$ unknowns.
\begin{proposition}[Cor. III.4.5.2
  of~\cite{bourbaki}] \label{hensel-system} Consider a collection of
  $n$ restricted power series $f_j(x_1,x_2,\dots,x_n)$ for $1 \le j\le
  n$ in $\mathbb{Z}_p[[x_1,x_2,\dots,x_n]]$. Let $(a_1,a_2,\dots,a_n)$ be
  a vector in $\mathbb{Z}_p^n$ such that the determinant of the
  Jacobian matrix at $(a_1,a_2,\dots,a_n)$
$$ \left|  {\partial(f_1,f_2,\dots,f_n) \over \partial(x_1,x_2, \dots,x_n)}  
(a_1,a_2,\dots,a_n) \right|$$
is in $\mathbb{Z}_p^\times$ and $f_j(a_1,a_2,\dots,a_n) 
\equiv 0 \pmod p$ for $1\le j \le n$. Then there exists a unique
$(x_1, x_2,\dots,x_n) \in \mathbb{Z}_p^n$ for which $x_i \equiv a_i \pmod p$ 
for $1\le i \le n$ and $f_j(x_1,x_2,\dots,x_n) =0$ in $\mathbb{Z}_p$ for $1 \le j \le n$.
\end{proposition}

As a corollary we get a generalization of one of the standard
formulations of Hensel's Lemma to the case of restricted power series.
\begin{corollary} \label{hensel-one}
Let
  $f(x)$ be a restricted power series in
  $\mathbb{Z}_p[[x]]$ and ${a}$ be in $\Zp$ such that
  $\frac{df}{dx}(a)$ is in $\Zp^\times$ and $f(a) \equiv 0 \pmod p$.
  Then there exists a unique
$x \in \mathbb{Z}_p$ for which $x \equiv a \pmod p$ 
and $f(x)=0$ in $\mathbb{Z}_p$.
\end{corollary}

In our discussion of collisions below we will also need a ``lifting
lemma'' for restricted power series of more than one variable which
will allow us to count solutions modulo higher powers of $p$ if we
know the number of solutions modulo $p$.  The following proposition,
which the second author learned from Igusa's 1986 ``Automorphic
Forms'' class at Johns Hopkins, is a generalization of the version of
Hensel's Lemma in Lemma~III.2.5 of~\cite{igusa-forms} to the case of
restricted power series, with explicit counting of the fibers.

\begin{proposition} \label{hensel-count} Let $f(x_1, x_2,\dots,x_n)$
  be a restricted power series in $\mathbb{Z}_p[[x_1,\dots,x_n]]$.
  Let $$N_e= \{ {\bf \bar{a}} \in (\Zp/p^e\Zp)^n \mid {\partial f
    \over \partial x_i} ({\bf{a}}) \in \Zp^\times \mbox { for some } 1
  \le i \le n \mbox{ and } f({\bf a}) \equiv 0 \pmod {p^e} \}$$ for
  $e>0$, where $\mathbf{\bar{a}}$ indicates reduction of $\mathbf{a}$
  to the appropriate residue class. Then $\psi: N_{e+1} \to N_e$ is a
  well-defined canonical surjection with the cardinality of the fiber
  equal to $p^{n-1}$.

  In particular, a point ${\bf \bar{a}}=(a_1, a_2,\dots,a_n) \in N_e$
  can be lifted in $p^{n-1}$ different ways to a point ${\bf \bar{b}}=(b_1,
  b_2,\dots,b_n) \in N_{e+1}$ such that $b_i \equiv a_i \pmod {p^e}$
  for $ 1 \le i \le n$, so that the relationship between the
  cardinalities of the sets is: $|N_{e+1}|= p^{n-1} |N_e|$ for $e>0$.
\end{proposition}

\section{Fixed Points} \label{fixed}

\begin{theorem} \label{x0fixedpt} For $p \neq 2$, let $g \in \Zp^\times$ be fixed and let $m$ be the
  multiplicative order of $g$ modulo $p$.
  Then for every $x_0 \in \Z/m\Z$, there is exactly one
  solution to the equation
$$\omega(g)^{x_0} \oneunit{g}^x = x $$
for $x \in \Zp$.
\end{theorem}

\begin{proof} We start by finding solutions modulo $p$.  We know that
  $\oneunit{g} \equiv 1 \pmod{p}$, so the equation reduces to 
$$\omega(g)^{x_0} \equiv x \pmod{p}.$$
For fixed $g$ and $x_0$, this obviously has exactly one solution.

Since we know that
  $\oneunit{g}$ is in $1+p\Zp$, we have
  that
\begin{eqnarray*}
\oneunit{g}^{x}=\exp(x \log(\oneunit{g}))=1&+&x\log(\oneunit{g})+
x^2\log( \oneunit{g})^2/2! \\
&+& \mbox{higher order terms in powers of }
\log(\oneunit{g})
\end{eqnarray*}
where from the definition of the $p$-adic logarithm we know that 
$\log(\oneunit{g}) \in p\Zp$.  Therefore we have a restricted
power series and we can apply Corollary~\ref{hensel-one}, which gives us a unique solution in $\Zp$.
\end{proof}

\begin{corollary} \label{crtfixedpt} For $p \neq 2$, let $g \in \Z$ be fixed such that $p \nmid g$ and let $m$ be the
  multiplicative order of $g$ modulo $p$.
  Then there are exactly $m$ solutions to the congruence
\begin{equation}
g^x \equiv x \pmod{p^e} \tag{\ref{fpe}, recalled}
\end{equation}
for $x \in \{1,2, \ldots, p^em \}$.  Furthermore, these solutions are
all distinct modulo $p^e$ and all distinct modulo $m$.
\end{corollary}

\begin{proof}
Theorem~\ref{x0fixedpt} implies that for each choice of $x_0 \in \Z/m\Z$ there is exactly one
$x_1 \in \Z/p^e\Z$ 
with the property that
$$\omega(g)^{x_0} \oneunit{g}^{x_1} \equiv x_1 \pmod {p^e}. $$
By the Chinese Remainder Theorem, there will be exactly one  $x \in \Z/p^em\Z$ such that
$x \equiv x_0 \pmod m$ and $x \equiv x_1 \pmod {p^e}$. By the interpolation set up since $x \equiv x_0 \pmod m$, we know that
for this $x$:
$$g^x =\omega(g)^{x_0} \oneunit{g}^{x} \equiv x \pmod {p^e}. $$
Finally, since exactly one such $x$ exists for each $x_0$, we have our $m$ solutions to the congruence.

\end{proof}

\section{Two-Cycles} \label{twocycles}

\begin{definition} 
For a fixed prime  $p$ and for some $g \in \Z$, $p \nmid g$, 
the pair $(h,a) \in \set{1, \ldots, p^e(p-1)}^2$, $p \nmid h$, $p \nmid
a$ will be a 
\emph{two-cycle modulo $p^e$ associated with $g$} if 
$h \not \equiv a \pmod {p^e}$, and
\begin{equation}
  g^{h} \equiv a \mod{p^e} \quad \text{and} \quad g^{a} \equiv h
  \mod{p^e}. \tag{\ref{tce}, recalled}
\end{equation}
\end{definition}

\begin{definition}
When we count \emph{the number of two-cycles modulo $p^e$}, we will not 
distinguish between the two-cycle $(h,a)$ and the two-cycle $(a,h)$. 
Thus, we define the number of two-cycles modulo $p^e$, or $|T_e|$, as

\begin{equation*}
\begin{split}
  |T_e|= \frac{1}{2}\ \Bigl| \Bigl\{ h &\in \set{1, \dots, p^e(p-1)}, p \nmid h
  \quad \mid  \\
& h \not \equiv a \pmod {p^e},  \quad
    g^h \equiv a \pmod {p^e}, \ \mbox{ and}  \quad  g^a \equiv h \pmod {p^e} \\
    &  \quad  \text{\ for some\ } g \in (\Z/p^e\Z)^\times \text{\ and\ }
    a \in \set{1, \ldots, p^e(p-1)}, p \nmid a \Bigr\} \Bigr|.
  \end{split}
\end{equation*}
\end{definition}

\begin{proposition} \label{x0y0tc}
  For $p \not = 2$ and a fixed $g \in \Z_p^\times$, let $m$ be the
  multiplicative order of g modulo $p$. Then for every pair $(x_0, \
  y_0) \in (\Z/m\Z)^2$, there is exactly one solution to the system of
  equations
\begin{eqnarray*}
\omega(g)^{x_0} \oneunit{g}^h &=& a \\
\omega(g)^{y_0} \oneunit{g}^a &=& h 
\end{eqnarray*}
for $(h,a) \in \Zp^2$.
\end{proposition}

\begin{proof}
We start by finding solutions modulo $p$. If we let 
\begin{eqnarray*}
f_1(h,a)&=&\omega(g)^{x_0} \oneunit{g}^h - a \\
f_2(h,a)&=&\omega(g)^{y_0} \oneunit{g}^a - h 
\end{eqnarray*}
then modulo $p$ this system reduces to
\begin{eqnarray*}
f_1(h,a)&\equiv&\omega(g)^{x_0}  - a  \pmod p\\
f_2(h,a)&\equiv&\omega(g)^{y_0}  - h \pmod p
\end{eqnarray*}
which clearly has exactly one solution $(h,a)=(\omega(g)^{x_0},\omega(g)^{y_0})$ for fixed $g$, $x_0$ and $y_0$.
The power series representations for $f_1(h,a)$ and $f_2(h,a)$ are restricted power series with
\begin{eqnarray*}
{\partial f_1 \over \partial h} &=& \omega(g)^{x_0}( \ \log(\oneunit{g}) + h \log(\oneunit{g})^2 + \cdots ) \equiv 0 \pmod p \\
{\partial f_1 \over \partial a} &=& -1 \equiv -1 \pmod p \\
{\partial f_2 \over \partial h} &=& -1 \equiv -1 \pmod p \\
{\partial f_2 \over \partial a} &=& \omega(g)^{y_0}( \ \log(\oneunit{g}) + a \log(\oneunit{g})^2 + \cdots ) \equiv 0 \pmod p.
\end{eqnarray*}

Thus the determinant of the Jacobian matrix is congruent to -1 modulo
$p$ and by Proposition~\ref{hensel-system} the unique solution modulo
$p$ to this system lifts to a unique solution $(h,a) \in \Zp^2$.
\end{proof}

\begin{proposition} \label{crt-tc} For $p \not = 2$ and a fixed $g \in
  \Z$, $p \nmid g$, let $m$ be the multiplicative order of $g$ modulo
  $p$.  Then if
\begin{equation*}
\begin{split}
  |T_{e,g}|= \frac{1}{2}\ \Bigl| \Bigl\{ h &\in \set{1, \dots, p^em}, p \nmid h
  \quad \mid \quad   h \not \equiv a \pmod {p^e},  \\ 
    & g^h \equiv a \pmod {p^e}, \ \mbox{ and}  \quad  g^a \equiv h \pmod {p^e} \\
    &  \quad  \text{\ for some\ } 
    a \in \set{1, \ldots, p^em}, p \nmid a \Bigr\} \Bigr|.
  \end{split}
\end{equation*}
is the number of two-cycles modulo $p^e$ associated with that
particular $g$,
$$
|T_{e,g}|=(m^2 -m)/2.
$$
\end{proposition}
\begin{proof}
Parallel to the proof of Corollary~\ref{crtfixedpt}, for each choice
of $(x_0, y_0)$ in $(\Z/m\Z)^2$, Proposition~\ref{x0y0tc} gives us exactly
one pair $(h,a)$ in $(\Z/p^em\Z)^2$ satisfying $ g^h \equiv a \pmod
{p^e}$ and $g^a \equiv h \pmod {p^e}$.  Thus there are $m^2$ such
pairs total, but $m$ of them correspond to the case where $h \equiv a
\pmod{p^e}$.  Dividing by 2 to account for swapping the roles of $h$
and $a$ gives us the proposition.
\end{proof}

\begin{theorem} For a given prime $p \not =2$, the number of
  two-cycles $|T_e|$ is
\begin{eqnarray*}
|T_e|= \sum_{m \mid (p-1)}\phi \left({m}\right)p^{e-1} (p-1) (m-1)/2.
\end{eqnarray*}
\end{theorem}
\begin{proof}
First note that if an $h$ in $\set{1, \dots, p^em}$ forms part of a
two-cycle associated with $g$ and $a$, then the values in
$\set{1, \dots, p^e(p-1)}$ which do the same will be exactly those which
are congruent to $h$ modulo $p^e$ and modulo $m$, and thus modulo
$p^em$.  So each element of $T_{e,g}$ gives rise to exactly $(p-1)/m$
elements of $T_e$ in this fashion.  On the other hand, if some $a$ in
$\set{1, \dots, p^e(p-1)}$ 
forms part of a two-cycle associated with $h$ and $g$, then so will an
$a$ in $\set{1, \dots, p^em}$ which is congruent to it modulo $p^em$.
So each element of $T_{e,g}$ gives rise to only one element of $T_e$
in this fashion. Therefore we have
\begin{align*}
|T_e| &= \sum_{g \in (\Z/p^e\Z)^\times} \left(\frac{p-1}{m}\right) \left|T_{e,g}\right| \\
&=\sum_{m \mid (p-1)}\phi \left({m}\right)p^{e-1} (p-1) (m-1)/2.
\end{align*}
\end{proof}

Alternatively, we can count rooted closed walks rather than cycles, a viewpoint
which in some ways lends itself better to generalizations.

\begin{definition} For a fixed prime $p$ and for some $g \in
  \Z$, $p \nmid g$, the ordered tuple $(h_1, \ldots, h_k)$ is a \emph{rooted
    closed walk of length $k$ modulo $p^e$ associated with $g$} 
if the $k$
  equations
\begin{align*}
  g^{h_1} &\equiv h_2 \mod{p^e},\\
  g^{h_2} &\equiv h_3 \mod{p^e},\\
 &\vdots\\
 g^{h_{k-1}} &\equiv h_k \mod{p^e},\\
 g^{h_k} &\equiv h_1 \mod{p^e}
\end{align*}
are satisfied.
\end{definition}

Then Corollary~\ref{crtfixedpt} is equivalent to saying that there are
exactly $m$ rooted closed walks of length $1$ associated with $g$ in
$\{1,2, \ldots, p^em \}$, and Proposition~\ref{crt-tc} is equivalent
to saying that there are $m^2$ rooted closed walks of length $2$
associated with $g$
(including the fixed points) in $\{1,2, \ldots, p^em \}^2$.  In an
exactly parallel manner, we can prove the following generalization:

\begin{theorem} 
  For $p \not = 2$ and a fixed $g \in \Z$, $p \nmid \Z$, let $m$ be
  the multiplicative order of $g$ modulo $p$.  Then there are exactly
  $m^k$ rooted closed walks of length $k$ modulo $p^e$ associated with
  $g$ in $\{1,2, \ldots, p^em \}^k$.  Furthermore, any two of these
  rooted closed walks are distinct modulo $p^e$ and distinct modulo
  $m$.
\end{theorem}

\begin{remark}
  In the case where $m=p-1$, this is an equivalent statement to
  Theorem~1 of~\cite{glebsky}, where it is proved using purely
  combinatorial methods.  For general $m$, our statement implies that
  of~\cite{glebsky}.
\end{remark}

\section{Self-Power Solutions} \label{selfpower}
We now turn to the function $x \mapsto x^x \bmod{p}$, which is
sometimes known as the \emph{self-power map}.  

The proof of the following elementary lemma was essentially worked out
in Theorem~2 of~\cite{field_et_al}.

\begin{lemma} \label{x0selfpowermodp} For $p \neq 2$, let $c \in
  (\Z/p\Z)^\times$ be fixed and let $m$ be the multiplicative order of
  $c$ modulo $p$.  Also fix $x_0 \in \{0,1, \ldots, p-2 \}$.
  Then the number of solutions $x \in
  (\Z/p\Z)^\times$ to the equivalence
$$
x^{x_0} \equiv c \pmod {p}
$$
is 
$$\displaystyle \begin{cases}
 \gcd(x_0, p-1)& \text{if $ \gcd(x_0, p-1) \mid \frac{p-1}{m}$;}\\
0 & \text{otherwise.}
\end{cases}$$
\end{lemma}

\begin{proof}
  For a fixed integer $t$ the set of $t$-th powers, $P_t=\{x^t \mid x \in
  (\Z/p\Z)^\times\}$ forms a subgroup of index $\gcd(t,p-1)$ in
  $(\Z/p\Z)^\times$. Using our set cardinality notation, we have that
  $|P_t|= (p-1)/\gcd(t,p-1)$.  If $ \gcd(x_0, p-1) \nmid
  \frac{p-1}{m}$ then $c$ is not in $P_{x_0}$, so $x^{x_0} \equiv c
  \pmod {p}$ cannot have any solutions.  Otherwise, any element of
  $P_{x_0}$ is an $x_0$-th power in exactly $ \gcd(x_0, p-1)$ ways, so
  the equivalence has exactly $ \gcd(x_0, p-1)$ solutions.
\end{proof}

\begin{corollary} \label{selfpowerS1}
For $p \neq 2$, let $c \in
  (\Z/p\Z)^\times$ be fixed and let $m$ be the multiplicative order of
  $c$ modulo $p$. Then the number of solutions $x \in
  \{1,2, \ldots, p(p-1)\}$ to the equivalence $x^x \equiv c \pmod{p}$
such that $p \nmid x$
is given by the formula:
$$
\sum_{\substack{0\le x_0 \le p-2 \\ \gcd(x_0, p-1) \mid
    \frac{p-1}{m}}} \gcd(x_0,p-1)= \sum_{d \mid \frac{p-1}{m}}
d \ \phi\left(\frac{p-1}{d}\right).$$
\end{corollary}

\begin{proposition} \label{x0spN1} 
For $p \neq 2$, let $c \in
  \Zp^\times$ be fixed and let $m$ be the multiplicative order of
  $c$ modulo $p$. 
Then for fixed $x_0 \in \Z/{(p-1)}\Z$, 
the number of solutions to the equation
$$\omega(x)^{x_0} \oneunit{x}^x =c$$
 for $x \in \Zp^\times$ 
is
$$\begin{cases}
\gcd(x_0,p-1)& \text{if $\gcd(x_0,p-1) \mid \frac{p-1}{m}$;}\\
0 & \text{otherwise.}
\end{cases}$$
\end{proposition}

\begin{proof} For a fixed $x_0$, we consider the function
$$f(x)=\omega(x)^{x_0} \oneunit{x}^x - c$$
and look for solutions 
$x \in \Zp^\times$ to $f(x) \equiv 0 \pmod{p}$. Since we know that $\oneunit{x}$
is in $1+p\Zp$, we have that
\begin{eqnarray*}
\oneunit{x}^x=\exp(x \log(\oneunit{x}))=1&+&x\log(\oneunit{x})+
x^2\log( \oneunit{x})^2/2! \\
&+& \mbox{higher order terms in powers of } x\log(\oneunit{x})
\end{eqnarray*}
where from the definition of the $p$-adic logarithm we know that 
$\log(\oneunit{x}) \in p\Zp$.
Now if we consider the power 
series representation of $f(x)$, we see that 
\begin{eqnarray*} \label{sppowsercoll}
f(x)= \omega(x)^{x_0} - c &+& \omega(x)^{x_0} x \log(\oneunit{x}) \\
&+& \mbox{ higher order terms in } p^2\Zp.
\end{eqnarray*}

Since $\omega$ is constant on each of the $p-1$ disjoint cosets of
$p\Zp$ that cover $\Zp^\times$ or see  
\cite[Prop.2, Section IV.2]{koblitz}, we have that
$$
\frac{df}{dx} = \omega(x)^{x_0} [ \log(\oneunit{x}) +1] \equiv
\omega(x)^{x_0} \pmod p 
$$
since $\log(\oneunit{x}) \in p\Zp$. As $\omega(x)^{x_0} \not \equiv 0
\pmod p$, we have by Corollary~\ref{hensel-one} that
the number of solutions in $\Zp$ is the same as the number of solutions in
Lemma~\ref{x0selfpowermodp}.
\end{proof}

\begin{corollary} \label{selfpowerSe}
For $p \neq 2$, let $c \in
  \Zp^\times$ be fixed and let $m$ be the multiplicative order of
  $c$ modulo $p$.   Then the number of solutions to the
  congruence
\begin{equation}
x^x \equiv c \pmod{p^e} \tag{\ref{spe}, recalled}
\end{equation}
for $x$ such that $x \in \{1,2,\ldots p^e(p-1)\}$, $p\nmid x$,
is given by the formula:
$$
\sum_{\substack{0\le x_0 \le p-2 \\ \gcd(x_0, p-1) \mid
    \frac{p-1}{m}}} \gcd(x_0,p-1)= \sum_{d \mid \frac{p-1}{m}}
d \ \phi\left(\frac{p-1}{d}\right).$$
\end{corollary}
\begin{proof} The proof is parallel to that of Corollary~\ref{crtfixedpt}.
\end{proof}

\section{Collisions} \label{collisions}

\begin{definition}
The set of solutions $(h,a)$, where $h$ and $a \in \{1,2,\ldots
p(p-1)\} $, $p \nmid h$ and $p\nmid a$, to the equivalence 
$$
h^h \equiv a^a \pmod p
$$
will be denoted $C_1$ for \emph{collisions} and we will use the
notation $|C_1|$ for the number of such collisions.  More generally,
we will use the notation $|C_e|$ to denote the number of
\emph{collisions} $(h,a)$, where $h$ and $a \in \{1,2,\ldots
p^e(p-1)\}$, $p \nmid h$ and $p\nmid a$, which are solutions to the
equivalence
$$
h^h \equiv a^a \pmod {p^e}.
$$ 
\end{definition}

Recall that $\bar{x}$ indicates reduction of $x$ to the appropriate
residue class.

\begin{lemma} \label{number} For fixed $x_0$ and $y_0 \in \{0,1,
  \ldots, p-2 \}$, if 
$$N_1^\times=\{({x},{y}) \in
  ((\Z/p\Z)^\times)^2 \mid x^{x_0}-y^{y_0} = 0 \text{\ in\ } \Z/p\Z \},$$
  then $$|N_1^\times|= (p-1)\gcd(x_0,y_0,p-1).$$
\end{lemma}

\begin{proof}
For a fixed integer $t$ the set of $t$-th powers, $P_t=\{x^t \mid x \in (\Z/p\Z)^\times\}$ forms a subgroup 
of index $\gcd(t,p-1)$ in $(\Z/p\Z)^\times$. Using our set cardinality notation, we have that
$|P_t|= (p-1)/\gcd(t,p-1)$.
Let $\mathfrak{I}=P_{x_0} \bigcap P_{y_0}$, then $\mathfrak{I}$ is a subgroup of order 

$$
|\mathfrak{I}|=\gcd(|P_{x_0}|, |P_{y_0}|)={ \frac {(p-1) \gcd(x_0,y_0,p-1) }  {\gcd(x_0,p-1) \gcd(y_0,p-1)}}.
$$

Now, we need to count all $(x,y) \in ((\Z/p\Z)^\times)^2$ such that $x^{x_0} \equiv y^{y_0} \pmod p$. 
If $x^{x_0} \equiv y^{y_0} \pmod p$ then $x^{x_0}$ and $y^{y_0} $ are in the set $\mathfrak{I}$ above.
Thus, we have that 
$$|N_1^\times|= \sum_{ i \in \mathfrak{I}} 
\left|\{ x \in (\Z/p\Z)^\times \mid x^{x_0} \equiv i \pmod p\}\right| \ \cdot \ 
\left|\{ y \in(\Z/p\Z)^\times \mid y^{y_0}\equiv i \pmod p\} \right| . $$ 
For each $i \in \mathfrak{I}$, $ |\{ x \in (\Z/p\Z)^\times \mid x^{x_0}
\equiv i \pmod p\} |=\gcd(x_0,p-1)$. So that
$$|N_1^\times|=| \mathfrak{I}| \cdot \gcd(x_0,p-1) \cdot \gcd(y_0,p-1)=(p-1)\gcd(x_0,y_0,p-1).$$
\end{proof}

\begin{proposition} \label{x0y0collN1} For $p \neq 2$ and for fixed $x_0$ and  $y_0 \in \Z/{(p-1)}\Z$, 
if we consider the function
$f(h,a)=\omega(h)^{x_0} \oneunit{h}^h -  \omega(a)^{y_0}
\oneunit{a}^a$ for $h, a \in \Zp^\times$ and let
$$|N_1^\times|=\left| \{(\bar{h},\bar{a}) \in ((\Zp/p\Zp)^\times)^2 \mid f(h,a)
  \equiv 0 \pmod{p}\} \right|,$$ then

$$|N_1^\times|=(p-1)\gcd(x_0,y_0,p-1)$$
\end{proposition}

\begin{proof} For a fixed $x_0$ and $y_0$, we look for solutions 
$h, a \in \Zp^\times$ to $f(h,a) \equiv 0 \bmod p$. Since we know that $\oneunit{h}$
and $\oneunit{a}$ are elements in $1+p\Zp$, we have that
\begin{eqnarray*}
\oneunit{h}^h=\exp(h \log(\oneunit{h}))=1&+&h\log(\oneunit{h})+
h^2\log( \oneunit{h})^2/2! \\
&+& \mbox{higher order terms in powers of } h\log(\oneunit{h})
\end{eqnarray*}
where from the definition of the $p$-adic logarithm we know that 
$\log(\oneunit{h}) \in p\Zp$.
Now if we consider the number of solutions $|N_1^\times|$ using the power 
series representation of $f(h,a)$, we see that 
\begin{eqnarray} \label{powsercoll}
f(h,a)= \omega(h)^{x_0} -  \omega(a)^{y_0} + \mbox{ higher order terms in } p\Zp.
\end{eqnarray}
In this way, we see that 
$$|N_1^\times|=
| \{(\bar{h},\bar{a}) \in ((\Z/p\Z)^\times)^2 \mid
\omega(h)^{x_0}-\omega(a)^{y_0} \equiv 0 \pmod p\}|.
$$
From this expression and Lemma~\ref{number}, we have that
$$
|{N}_1^\times|= (p-1)\gcd(x_0,y_0,p-1).$$
\end{proof}

\begin{corollary} \label{collisionC1}
For $p \neq 2$, the number of collisions $(h, a)$
  for $h$ and $a \in \{1,2, \ldots, p(p-1)\}$ such that $p \nmid h$,
  $p \nmid a$, and $h^h \equiv a^a \pmod{p}$ is given by the formula:
$$
|\Nha_1|= \sum_{0\le x_0,y_0 \le p-2} (p-1) \gcd(x_0,y_0,p-1)=
(p-1) \sum_{d \mid p-1} d  \  J_2((p-1)/d)$$
where $J_2 (n) =n^2 \prod_{p \mid n} (1-p^{-2})$ is Jordan's totient
function, which counts the number of pairs of positive integers all
less than or equal to $n$ that form a mutually coprime triple together
with $n$. 
\end{corollary}

\begin{proposition} \label{x0y0collNe}For $p \neq 2$ and for fixed $x_0$ and  $y_0 \in \Z/{(p-1)}\Z$, 
if we consider the function
$f(h,a)=\omega(h)^{x_0} \oneunit{h}^h -  \omega(a)^{y_0} \oneunit{a}^a$ 
for $h, a \in \Zp^\times$ and let
$$N_e^\times=\{(\bar{h},\bar{a}) \in ((\Zp/p^e\Zp)^\times)^2  \mid f(h,a) \equiv 0 \pmod {p^e} \},$$ then
$$|N_e^\times|=
p^{e-1}|N_1^\times|.$$
\end{proposition}

\begin{proof}
Looking more carefully at our series representation for 
$f(h,a)$ in Equation~\ref{powsercoll} from Proposition~\ref{x0y0collN1}, we have that
\begin{eqnarray*} \label{powsercoll2}
f(h,a)= \omega(h)^{x_0} -  \omega(a)^{y_0} & + & \omega(h)^{x_0} h \log(\oneunit{h}) - \omega(a)^{y_0} a \log(\oneunit{a}) \\
&+& \mbox{ higher order terms in } p^2\Zp.
\end{eqnarray*}
Since $\omega$ is constant on each of the $p-1$ disjoint cosets of $p\Zp$ that cover $\Zp^\times$ or see 
\cite[Prop.2, Section IV.2]{koblitz}, we have that
$$
{\partial f \over \partial h}= \omega(h)^{x_0} [ \log(\oneunit{h}) +1] \equiv  \omega(h)^{x_0} \pmod p
$$
since $\log(\oneunit{h}) \in p\Zp$. As $\omega(h)^{x_0} \not \equiv 0
\pmod p$, we have by Proposition~\ref{hensel-count} with $n=2$ that
$$
|{N}_e^\times|=
p|N_{e-1}^\times|.
$$
for $e >1$, and our Proposition follows.
\end{proof}

\begin{corollary} \label{collisionCe}
For $p \neq 2$, there are exactly $| \Nha_e|=p^{e-1} |{\Nha}_1|$ collisions that are 
solutions to the congruence
\begin{equation}
h^h \equiv a^a \pmod{p^e} \tag{\ref{hae}, recalled}
\end{equation}
for $(h, a)$ such that $h$ and $a \in \{1,2,\ldots p^e(p-1)\}$, $p\nmid h$, $p\nmid a$.
\end{corollary}
\begin{proof} The proof is parallel to that of Corollary~\ref{crtfixedpt}.
\end{proof}

\begin{remark}
Note that Corollaries~\ref{collisionC1} and~\ref{collisionCe} could
also have been proved by squaring the results of
Corollaries~\ref{selfpowerS1} and~\ref{selfpowerSe}, respectively, and
summing over all $c$.
\end{remark}

\section{Conclusions and Future Work}

Previous work on solutions to~\eqref{fp} and related equations has
focused on finding how primitive roots modulo $p$, or specified powers
of primitive roots, are distributed in arithmetic progressions
contained in $\set{1, \ldots, p}$ with differences dividing $p-1$.  We
hope that this paper shows that another course might also be fruitful:
start with the solutions to an exponential equation which are in
$\set{1, \ldots, p(p-1)}$ (or $\set{1 \ldots, p^e(p-1)}$) and
determine how they are distributed among the subintervals of length
$p$ (or $p^e$).  Furthermore, we think the use of $p$-adic numbers
also suggests new lines of attack that may be useful in the future.
For example, the ability to extend the $p$-adic exponential function
to rings of integers in extension fields of $\Qp$ might provide a
useful way of looking at, or even posing, new problems in finite field
extensions of $\Z/p\Z$.

In future extensions of this work we hope to consider solutions of
more exponential equations, including the equation 
\begin{equation} 
    h^{h/d} \equiv a^{a/d} \mod{p^e}, \quad d = \gcd(h, a, p-1)
\end{equation}
considered (with $e=1$) in~\cite{holden_moree} as closely related to~\eqref{tc}.
Another problem that should be tractable using our
methods is finding solutions of
\begin{equation}
g^{x-1+c} \equiv x \mod{p^e}
\end{equation}
for $c$ fixed.  This was raised in~\cite{drakakis} (with $e=1$) as
related to ``Golumb rulers'', which have applications in error
correction and in controlling the effects of electromagnetic signals
interference.  Finally, one could consider the ``discrete Lambert''
map $x \mapsto xg^x$ for $g$ fixed, which is related to the standard
ElGamal signature scheme and the Digital Signature Algorithm in a
similar fashion to the way the self-power function is related to its
variants.  Then one could ask for solutions of
\begin{equation}
xg^x \equiv c \mod{p^e}
\end{equation}
for fixed $c$, or collisions of the discrete Lambert map, namely
solutions of 
\begin{equation}
    hg^{h} \equiv ag^{a} \mod{p^e} .
\end{equation}

Finally, for completeness one should investigate the situation when
$p=2$.  Counting solutions modulo $p$ is trivial in this case, but the
$p$-adic situation is slightly more complicated than in the $p \neq 2$
case.

\end{document}